\documentclass[preprint,12pt]{elsarticle}
\usepackage[cp1251]{inputenc}
\usepackage{graphicx}
\usepackage{epsfig}

\usepackage{amssymb}
\usepackage{amsthm}
\newcommand{\sysn}{\left\{\begin{array}{rcl}}
\newcommand{\sysk}{\end{array}\right.}

\newtheorem{theorem}{Theorem}[section]

\newtheorem{proposition}[theorem]{Proposition}
\theoremstyle{definition}

\newtheorem{corollary}[theorem]{Corollary}

\journal{...}

\begin{document}

\begin{frontmatter}


%%\title{Topological-Algebraic Properties of Function Space with Set-Open Topology\tnoteref{label1}}
%%\tnotetext[label1]{}
%%\author{Alexander V. Osipov\corref{cor1}\fnref{label2}}
%%\ead{OAB@list.ru}

%\title{Solution of a Problem of "Scottish Book" \\ for Banach spaces}

\title{Note on the Banach Problem 1 of condensations of Banach spaces onto compacta}


%\author[label1]{Taras Banakh}

%\ead[label1]{...}

%\tnotetext[label1]{The research has been supported by .}

%\address[label1]{Ivan Franco National University of Lviv, Ukraine, and Jan Kochanowski University in Kielce, Poland}

\author{Alexander V. Osipov}

\ead{OAB@list.ru}

%\tnotetext[label1]{The research has been supported by .}

\address{Krasovskii Institute of Mathematics and Mechanics, Ural Federal
 University, Ural State University of Economics, Yekaterinburg, Russia}

\begin{abstract}

It is consistent with any possible value of the continuum
$\mathfrak{c}$ that every infinite-dimensional Banach space of
density $\leq \mathfrak{c}$ condenses onto the Hilbert cube.

Let $\mu<\mathfrak{c}$ be a cardinal of uncountable cofinality. It
is consistent that the continuum be arbitrary large, no Banach
space $X$ of density $\gamma$, $\mu<\gamma<\mathfrak{c}$,
condenses onto a compact metric space, but any Banach space of
density $\mu$ admits a condensation onto a compact metric space.
In particular, for $\mu=\omega_1$, it is consistent that
$\mathfrak{c}$ is arbitrarily large, no Banach space of density
$\gamma$, $\omega_1<\gamma<\mathfrak{c}$, condenses onto a compact
metric space.

These results imply a complete answer to the Problem 1 in the
Scottish Book for Banach spaces: When does a Banach space $X$
admit a bijective continuous mapping onto a compact metric space?
\end{abstract}


\end{frontmatter}

% \linenumbers


\section{Introduction}

 The following problem is a reformulation of the well-known
 problem of Stefan Banach from the Scottish Book:

{\bf Banach Problem.} When does a metric (possibly Banach) space
$X$ admit a condensation (i.e. a bijective continuous mapping)
onto a compactum (= compact metric space) ?

M. Katetov \cite{kat} was one of the first who attacked the Banach
problem. He proved that: a countable regular space has a
condensation onto a compactum if, and only if, it is scattered (a
space is said to be {\it scattered} if every nonempty subset of it
has an isolated point).
\medskip

Recall that a topological space is {\it Polish} if $X$ is
homeomorphic to a separable complete metric space and a
topological space $X$ is {\it $\sigma$-compact} if $X$ is a
countable union of compact subsets.

In 1941, A.S. Parhomenko \cite{parh2} constructed a example of a
$\sigma$-compact Polish space $X$ such that $X$ does not have a
condensation onto a compact space.

Replacing \cite{kad} in this proof by Kadets' homeomorphism
theorem \cite{kad1} implies that every infinite-dimensional
separable Banach space condenses onto the Hilbert cube.

\medskip

Recall that a space $X$ is called {\it absolute Borel}, if $X$ is
homeomorphic to a Borel subset of some complete metrizable space.

\medskip

In 1976, E.G. Pytkeev \cite{pytk} proved the following remarkable
theorem for separable absolute Borel non-$\sigma$-compact spaces.

\begin{theorem} Every separable absolute Borel space
$X$ condenses onto the Hilbert cube, whenever $X$ is not
$\sigma$-compact.
\end{theorem}

\medskip

 The Pytkeev's result implies that every separable complete non-$\sigma$-compact metric space condenses onto the Hilbert cube.
Thus every separable complete linear metric space admits a
condensation onto a compactum.

It is well known that any locally compact admits a condensation
onto a compact space (Parhomenko's Theorem) \cite{parh2}. Hence,
each separable metrizable locally compact space (and thus each
finite-dimensional Banach space) condenses onto a compactum. Thus
every separable Banach space admits a condensation onto a
compactum.

The {\it density} $d(X)$ of a topological space $X$ is the
smallest cardinality of a dense subset of $X$. Since metrizable
compact spaces have cardinality at most continuum, every metric
space admitting a condensation onto a compactum has density at
most continuum.

T.Banakh and A.Plichko \cite{bp} proved the following interesting
result.

\begin{theorem}\cite{bp}\label{thBP1}
Every Banach space $X$ of density $\aleph_1$ or $\mathfrak{c}$
admits a condensation onto the Hilbert cube.
\end{theorem}

\medskip

{\bf Question.} {\it What about intermediate densities between
$\aleph_1$ and $\mathfrak{c}$ ?}

\medskip

 It is clear that the answer to this question depends on a set theory model
independent of ZFC.

\medskip

In \cite{banakh}, T.Banakh announces the following results:

(1) It is consistent that the continuum is arbitrarily large and
every infinite-dimensional Banach space of density $\leq
\mathfrak{c}$ condenses onto the Hilbert cube $[0,1]^{\omega}$.

(2) It is consistent that the continuum is arbitrarily large and
{\bf no} Banach space of density $\aleph_1<d(X)<\mathfrak{c}$
condenses onto a compact metric space.

In this paper we give an independent proof of these results.

\section{Main results}

\begin{theorem}\label{thBP}(\cite{banakh}) If for some infinite
cardinal $\kappa$ there is  a partition of real line by $\kappa$
many Borel sets, then any Banach space of density $\kappa$
condenses onto the Hilbert cube.
\end{theorem}

In \cite{BrMi} (Theorem 3.8), W.R. Brian and A.W. Miller proved
the following result.

\begin{theorem}\label{thBM} It is consistent with any possible
value of $\mathfrak{c}$ that for every $\kappa\leq \mathfrak{c}$
there is a partition of $2^{\omega}$ into $\kappa$ closed sets.
\end{theorem}

\medskip

The following theorem is a mathematical folklore. It is a
corollary of Theorem \ref{thBP} and Theorem \ref{thBM}.

\begin{theorem}\label{thBO} It is consistent with any possible
value of $\mathfrak{c}$ that every infinite-dimensional Banach
space of density $\leq \mathfrak{c}$ condenses onto the Hilbert
cube.
\end{theorem}

\begin{proof} Because $\omega^{\omega}$ can be identified with a
co-countable subset of $2^{\omega}$, the model in Theorem
\ref{thBM} has, for every $\kappa<\mathfrak{c}$, a partition of
$\omega^{\omega}$ (and hence the real line, identifying
$\omega^{\omega}$ with irrational numbers) into $\kappa$ Borel
sets. It remain to apply Theorem \ref{thBP}.
\end{proof}

In fact the proof of Theorem \ref{thBO} (assuming the Brian-Miller
model of set theory) is a minor modification of the proof of Main
Theorem from \cite{bp}.

\medskip

Let $FIN(\kappa,2)$ be the partial order of finite partial
functions from $\kappa$ to $2$, i.e., Cohen forcing.

\begin{proposition}(Corollary 3.13 in \cite{BrMi})\label{prop} Suppose $M$ is
a countable transitive model of $ZFC+GCH$. Let $\kappa$ be any
cardinal of $M$ of uncountable cofinality which is not the
successor of a cardinal of countable cofinality.  Suppose that $G$
is $FIN(\kappa,2)$-generic over $M$, then in $M[G]$ the continuum
is $\kappa$ and for every uncountable $\gamma<\kappa$ if $F:
\gamma^{\omega}\rightarrow \omega^{\omega}$ is continuous and
onto, then there exists a $Q\in [\gamma]^{\omega_1}$ such that
$F(Q^{\omega})=\omega^{\omega}$.
\end{proposition}

Note that trivial modification to the proof of Proposition 3.14 in
\cite{BrMi} allow us to replace $\omega_2$ with any cardinal $\mu$
of uncountably cofinality.

\begin{proposition}\label{pr1} It is consistent that the continuum be
arbitrary large, $\omega^{\omega}$ can be partitioned into $\mu$
($\mu$ cardinal of uncountably cofinality) Borel sets, and
$\omega^{\omega}$ is not a condensation of $\kappa^{\omega}$
whenever $\mu<\kappa<\mathfrak{c}$.
\end{proposition}

\begin{theorem}\label{th3} Suppose $\mu$ is a cardinal of uncountable cofinality.  It is consistent that the continuum be arbitrary
large, no Banach space $X$ of density $\gamma$,
$\mu<\gamma<\mathfrak{c}$, condenses onto a compactum, but any
Banach space of density $\mu$ admit a condensation onto a
compactum.
\end{theorem}

\begin{proof} Suppose $M$ is
a countable transitive model of $ZFC+GCH$. Let $\kappa>\mu$ be any
cardinal of $M$ of uncountable cofinality which is not the
successor of a cardinal of countable cofinality.  Suppose that $G$
is $FIN(\kappa,2)$-generic over $M$, then in $M[G]$ the continuum
is $\kappa$ and for every uncountable $\gamma<\kappa$ if $F:
\gamma^{\omega}\rightarrow \omega^{\omega}$ is continuous and
onto, then there exists a $Q\in [\gamma]^{\mu}$ such that
$F(Q^{\omega})=\omega^{\omega}$ (Proposition \ref{prop} (Corollary
3.13 in \cite{BrMi}) with replacement $\omega_1$ with any cardinal
$\mu<\kappa$ of uncountably cofinality).

 By Proposition \ref{pr1}, $\omega^{\omega}$ can be partitioned
into $\mu$ Borel sets. By Theorem \ref{thBP}, any Banach space of
density $\mu$ admit a condensation onto the Hilbert cube
$[0,1]^{\omega}$.

 The proof of Theorem 3.7 in \cite{Miller} uses Cohen reals, but the same idea shows that this generic extension has the property that

$(\star)$ for every family $\mathcal{F}$ of Borel subsets of
$\omega^{\omega}$ with size $\mu<|\mathcal{F}|<\mathfrak{c}$, if
$\bigcup \mathcal{F}=\omega^{\omega}$ then there exists
$\mathcal{F}_0\in [\mathcal{F}]^{\mu}$ with $\bigcup
\mathcal{F}_0=\omega^{\omega}$ (see Proposition 3.14 in
\cite{BrMi}).

Let $\mu<\gamma<\mathfrak{c}$. It suffices to note that any Banach
space  of density $\gamma$  is homeomorphic to
$J(\gamma)^{\omega}$ where $J(\gamma)$ is hedgehog of weight
$\gamma$ (Theorem 5.1, Remark and Theorem 6.1 in \cite{tur}).

Let $f$ be a condensation from $\gamma^{\omega}$ onto
$J(\gamma)^{\omega}$ \cite{pytk}. Suppose that $g$ is a
condensation of $J(\gamma)^{\omega}$ onto a compact metric space
$K$. Then we have the condensation $h=g\circ
f:\gamma^{\omega}\rightarrow K$ of $\gamma^{\omega}$ onto $K$.

Let $\sum=[\gamma]^{\omega}\cap M$. Note that
$|\sum|<\mathfrak{c}$ since in $M$ $|\gamma^{\omega}|>\gamma$ if
and only if $\gamma$ has cofinality $\omega$, but in that case
$|\gamma^{\omega}|=|\gamma^+|<\mathfrak{c}$. Since the forcing is
c.c.c.

$M[G]\models \gamma^{\omega}=\bigcup \{Y^{\omega}: Y\in \sum \}$.

For any $Y\in \sum$ the continuous image $h(Y^{\omega})$ is an
analytic set (a $\Sigma^1_1$ set) and, hence the union of
$\omega_1$ Borel sets in $K$ (see Ch.3, $\S$ 39, Corollary 3 in
\cite{Kur1}), i.e.,  $h(Y^{\omega})=\bigcup\{B(Y,\beta):
\beta<\omega_1\}$ where $B(Y,\beta)$ is a Borel subset of $K$ for
each $\beta<\omega_1$. Note that $|\{B(Y,\beta): Y\in \sum,
\beta<\omega_1 \}|\leq |\sum|\cdot\aleph_1=|\sum|$.

Assume that $\theta=|\{B(Y,\beta): Y\in \sum, \beta<\omega_1 \}|<
\gamma$. Consider a function $\phi: \{B(Y,\beta): Y\in \sum,
\beta<\omega_1 \}\rightarrow \sum$ such that
$\phi(B(Y,\beta))=Y_{\xi}\in \sum$  where $h(Y_\xi^{\omega})$
contains in decomposition $B(Y,\beta)$ ($Y_{\xi}$ may be the same
for different $B(Y_1,\beta_1)$ and $B(Y_2,\beta_2)$). Then
$\bigcup\{Y_{\xi}:\xi\in\theta\}\in [\gamma]^{\leq \theta}$ and
$\gamma^{\omega}=\bigcup \{Y_{\xi}^{\omega}: \xi\in \theta \}$  is
a contradiction. Thus, $\gamma\leq \theta\leq
|\sum|<\mathfrak{c}$.

Since $K$ is Polish, there is a continuous surjection $p:
\omega^{\omega}\rightarrow K$. Given a family
$\mathcal{F}=\{p^{-1}(B(Y,\beta)): Y\in \sum, \beta<\omega_1 \}$
of $\theta$-many Borel sets
 ($\mu<\theta<\mathfrak{c}$) whose
union is $\omega^{\omega}$. By property $(\star)$, there is a
subfamily $\mathcal{F}_0=\{F_{\alpha}:
F_{\alpha}=p^{-1}(B(Y_{\alpha},\beta_{\alpha}))$, $\alpha<\mu\}$
of size $\mu$ whose union is $\omega^{\omega}$. Then the family
$\{h(Y^{\omega}_{\alpha}): \alpha<\mu\}$ of size $\mu$ whose union
is $K$. Let $Q=\bigcup\{Y_{\alpha}: \alpha<\mu\}$. Then $Q\in
[\gamma]^{\mu}$ and $h(Q^{\omega})=K$. Since $\mu<\gamma$, we
obtain a contradiction with injectivity of the mapping $h$.

\end{proof}

By Theorem \ref{th3} for $\mu=\omega_1$  we have the following
result.

\medskip
\begin{theorem}\label{th26} Suppose $M$ is a countable transitive model of $ZFC+GCH$. Suppose that $G$
is $FIN(\mathfrak{c},2)$-generic over $M$. No Banach space $X$ of
density $\gamma$, $\aleph_1<\gamma<\mathfrak{c}$ condenses onto a
compact metric space.
\end{theorem}

In \cite{Brain1}, W. Brian proved the following result.

\begin{theorem}\label{Th2} Let $\kappa<\aleph_{\omega}$, let $f:Y\rightarrow X$ be a condensation of a topological space $Y$
onto a Banach space $X$ of density $\kappa$. Then there is a
partition of $Y$ into $\kappa$ Borel sets.
\end{theorem}

\begin{theorem}\label{Os} Let $n<\omega$. The
following assertions are equivalent:

\begin{enumerate}

\item Any Banach space $X$ of density $\aleph_n$ condenses onto
the Hilbert cube;

\item $\omega^{\omega}$ can be partitioned into $\aleph_n$ Borel
sets;

\item $\omega^{\omega}$ is a condensation of $\omega_n^{\omega}$.

\end{enumerate}

\end{theorem}

\begin{proof} $(2)\Rightarrow(1)$. Since there is a
partition of $\omega^{\omega}$ into $\aleph_n$ Borel sets then, by
Theorem \ref{thBP} , any Banach space of weight $\aleph_n$ admit a
condensation onto the Hilbert cube.

$(1)\Rightarrow(2)$. Since $J(\aleph_n)^{\omega}$ is a
condensation of $\omega_n^{\omega}$ and  $J(\aleph_n)^{\omega}$
admit a condensation onto the Hilbert cube $[0,1]^{\omega}$ then
$[0,1]^{\omega}$ can be partitioned into $\aleph_n$ Polish sets
$B_{\alpha}$ (Theorem \ref{Th2}). Since $[0,1]^{\omega}$ is Polish
there is a continuous surjection $p: \omega^{\omega}\rightarrow
[0,1]^{\omega}$. Hence, $\omega^{\omega}$ can be partitioned into
$\aleph_n$ Borel sets $p^{-1}(B_{\alpha})$.

$(2)\Leftrightarrow(3)$. By Theorem 3.6 in \cite{BrMi}.

\end{proof}

 By Theorems \ref{th3} and \ref{Os} and Theorem 3.2 (and Corollaries 3.3 and 3.4) in \cite{WB} we have the
 following results for $\aleph_0<\kappa\leq \mathfrak{c}$.

\begin{corollary} Given any $A\subseteq \mathbb{N}$, there is a
forcing extension in which

1. Any Banach space $X$ of density $\kappa\in \{\aleph_n: n\in
A\}\cup
\{\aleph_1,\aleph_{\omega},\aleph_{\omega+1}=\mathfrak{c}\}$
condenses onto the Hilbert cube;

2. No Banach space $X$ of density $\kappa\not\in \{\aleph_n: n\in
A\}\cup
\{\aleph_1,\aleph_{\omega},\aleph_{\omega+1}=\mathfrak{c}\}$
condenses onto a compact metric space.

\end{corollary}

\begin{corollary} Given any finite $A\subseteq \mathbb{N}$, there is a
forcing extension in which

1. Any Banach space $X$ of density $\kappa\in \{\aleph_n: n\in
A\}\cup \{\aleph_1\}$ condenses onto the Hilbert cube;

2. No Banach space $X$ of density $\kappa\not\in \{\aleph_n: n\in
A\}\cup \{\aleph_1\}$ condenses onto a compact metric space.

\end{corollary}

{\bf Acknowledgements.} I would like to thank William Brian and
Taras Banakh for several valuable comments.

\bibliographystyle{model1a-num-names}
\bibliography{<your-bib-database>}

\end{document}